\documentclass[11pt]{article}
\usepackage[english]{babel}
\usepackage[utf8]{inputenc}		
\usepackage[T1]{fontenc}

\usepackage{amsmath,amsfonts,amssymb,amsthm,cancel,siunitx,
calculator,calc,mathtools,empheq,latexsym}
\usepackage{optidef}
\usepackage{subfig,epsfig,tikz,float}

\usepackage{booktabs,multicol,multirow,tabularx,array}
\usepackage{enumitem}
\usepackage{url}
\newtheorem{observation}{Observation}
\newtheorem{proposition}{Proposition}
\newtheorem{remark}{Remark}
\allowdisplaybreaks

\title{\normalsize\bf
\uppercase{Network Bandwidth Allocation Problem \\For Cloud Computing}
}
\author{
Changpeng Yang$^{2}$,
Jintao You$^{2}$, Xiaoming Yuan$^{1}$, Pengxiang Zhao$^{1*}$
}

\begin{document}

\date{}

\maketitle

\vspace{-0.5cm}

\begin{center}

*Corresponding author\\
$^1$Department of Mathematics, University of Hong Kong \\
$^2$Algorithm Innovation Lab, Huawei \\
E-mails: yangchangpeng@huawei.com, youjintao5@huawei.com, \\
xmyuan@hku.hk, pengxiangzhao@connect.hku.hk

\smallskip

\today

\end{center}

\bigskip
\noindent
{\small{\bf Abstract.}
Cloud computing enables ubiquitous, convenient, and on-demand network access to a shared pool of computing resources. Cloud computing technologies create tremendous commercial values in various areas, while many scientific challenges have arisen accordingly.
The process of transmitting data through networks is characterized by some distinctive characteristics such as nonlinear, nonconvex and even noncontinuous cost functions generated by pricing schemes, periodically updated network topology, as well as replicable data within network nodes. Because of these characteristics, data transfer scheduling is a very challenging problem both engineeringly and scientifically. On the other hand, the cost for bandwidth is a major component of the operating cost for cloud providers, and thus how to save bandwidth cost is extremely important for them to supply service with minimized cost.
We propose the Network Bandwidth Allocation (NBA) problem for cloud computing and formulate it as an integer programming model on a high level,
with which more comprehensive and rigorous scientific studies become possible. We also show that the NBA problem captures some of the major cloud computing scenarios including the content delivery network (CDN), the live video delivery network (LVDN), the real-time communication network (RTCN), and the cloud wide area network (Cloud-WAN).
}

\medskip
\noindent
{\small{\bf Keywords}{:}
cloud computing, bandwidth allocation, network, the 95th percentile billing, integer programming, content delivery network, live video delivery network, real-time communication, cloud wide area network
}

\baselineskip=\normalbaselineskip

\section{Introduction}\label{sec:1}
Nowadays, cloud computing is infrastructure of information technology industry; it has boosted some revolutionary technologies in various areas and essentially reshaped some economic ecosystems. The term ``cloud'' originates from the field of telecommunications, over which providers offer Virtual Private Network (VPN) service while redirect traffic to balance the load of the overall network is allowed, see \cite{jadeja2012cloud}.
Cloud computing extends the VPN service; it enables ubiquitous, convenient and on-demand network access to a shared pool of computing resources, which may include networks, servers, storage, and applications, see, e.g. \cite{mell2011nist}. Cloud computing is the pillar to many highly commercialized economic ecosystems, and it creates tremendous commercial values.

Customers, cloud providers, and Internet service providers (ISPs) are three major stakeholders in cloud computing.
For a cloud provider, it targets to supply reliable, customized, and quality of service (QoS) guaranteed service for customers, with the purpose of minimizing the bandwidth cost (paid to ISPs) which is indeed the main component of the overall operating cost, see \cite{armbrust2010view, liu2012}.
Therefore, scheduling data transmission to save bandwidth cost is extremely important for cloud providers.
In contrast to other traffic problems such as Vehicle Routing Problems, data transfer through computer networks has the following three distinctive characteristics:
\begin{itemize}[noitemsep]
    \item The cost functions generated by pricing schemes are nonlinear, nonconvex, and noncontinuous;
    \item Network topology updates periodically;
    \item Data can be replicated within network nodes.
\end{itemize}
The first feature is standard in the industry. For example, as mentioned in \cite{jiangjuncheng2015,zhan2016optimal,singh2021cost-effective}, the well-known $95^{th}$ percentile billing is a widely used pricing scheme, which leverages $95^{th}$ percentile of the bandwidth distribution over monthly periods.
Due to the variability of network status, for example, congestion may occur among some connections, network topology may update periodically to meet requirements such as reliability and low latency. This explains the second characteristic above.
As for the third one, digital data is a kind of special merchandise that can be replicated in devices without occupying transfer bandwidth.
Therefore, a network node, which is usually referred to a server or a data center, only needs to ask for desired data from another node while it can provide the obtained data to more than one node.
An underlying consequence is that the egress bandwidth of intermediary network nodes contained in the path of transmitting the same content is usually more than its ingress bandwidth.
To the best of our knowledge, it is the first time to consider this characteristic for modelling cloud computing problems.
With the concern of saving bandwidth cost in the industry as well as these important attributes of data transfer over computer networks, we propose the Network Bandwidth Allocation (NBA) problem for various cloud computing problems, and initiate the effort of studying it from the optimization perspective.

Formally speaking, the NBA problem is defined on a network $G = (V, E)$ during a billing cycle $P$, in which $V = \{1, 2, \dots, n\}$ is the set of network nodes abstracted from servers or data centers; $E = \{(i, j)|i, j \in V, i \neq j\}$ is the set of directed edges abstracted from network links; and $P$ is a given period of time separated by a set of sampling time points $T = \{1, 2, \dots, p\}$.
The egress bandwidth of the node $i \in V$ denoted by $\phi_i$ is the summation of bandwidth allocated to directed edges $\{(u, v)|u=i, v\in V, (u, v) \in E\}$; and its ingress bandwidth denoted by $\psi_i$ is the summation of bandwidth allocated to directed edges $\{(u, v)|u \in V, v=i, (u, v) \in E\}$.
We denote by $c_i^{out}$ and $c_i^{in}$ the admissible maximum egress bandwidth and maximum ingress bandwidth for the node $i \in V$, respectively.
During the billing cycle $P$, the pricing scheme is associated with the bandwidth distribution over $T$.
For the node $i \in V$, let $b_i^{out}$ be the $95^{th}$ percentile of its egress bandwidth distribution $\{\phi_i^{(t)}\}_{t=1}^p$; let $b_i^{in}$ be the $95^{th}$ percentile of its ingress bandwidth distribution $\{\psi_i^{(t)}\}_{t=1}^p$.
The cost incurred on $i$ is given by $u_i \cdot max(\{b_i^{out}, b_i^{in}\})$, where $u_i > 0$ is the unit-price and $max(\cdot)$ is the operation to get the maximum element of a number set.
Assume that in the time slot $[t, t+1], t \in T$, the set of available edges is $E^{(t)} \subseteq E$; the set of source nodes providing data is $S^{(t)} \subseteq V$; the set of destination nodes requiring data that is originally from the source node $s \in S^{(t)}$ is $D_s^{(t)} \subseteq V \setminus \{s\}$; and the bandwidth required to transmit such data with an acceptable latency between two nodes is at least $w_s^{(t)} > 0$.
The NBA problem consists of determining bandwidth allocation plans for the network $G$ to minimize the total bandwidth cost during $P$, while the following conditions are satisfied in every time slot $[t, t+1], t \in T$:
\begin{itemize}[noitemsep]
    \item Data can be transmitted out from all source nodes;
    \item Data can be transmitted into all destination nodes;
    \item The egress bandwidth allocation to a node does not exceed the admissible maximum egress bandwidth, and the ingress bandwidth allocation to it does not exceed the admissible maximum ingress bandwidth;
    \item To transmit data that is originally from the same source node, the required ingress bandwidth of a node is no more than the required egress bandwidth unless it is a destination node.
\end{itemize}
Note that the last condition takes the data replication ability of network nodes into account.
The NBA problem is an abstract and generic model, and it can be extended flexibly to suit various scenarios in practice.
The NBA problem is challenging because of its distinctive characteristics mentioned above, and also because of the high dimensionality of its variables in real applications.

In Section \ref{sec:2}, we review some related works. Then, we focus on the integer programming formulation of the NBA problem in Section \ref{sec:pro_for}.
In Section \ref{rea_ins}, we show some real-life cloud computing applications which can be captured by the NBA problem and its extensions.
Finally, some discussions are included in Section \ref{sec:dis}.

\section{Literature review}\label{sec:2}
To save bandwidth cost, it is conventional to apply a centralized controller to schedule data traffic in a network.
In \cite{hong2013}, a centralized software-driven wide area network (SWAN) system is proposed to coordinate the sending rates and network paths among data centers globally.
Besides, a centralized traffic engineering system named B4 is proposed in \cite{Sushant2013} to improve the utilization of Google's global data centers, in which traffic flows are split into multiple paths for balance.
To balance the requirements of decreasing latency and saving cost, a system called video delivery network (VDN) is proposed in \cite{jiangjuncheng2015} to allow cloud providers to control stream placement and bandwidth allocation dynamically.
The proposed VDN incorporates a central controller and local agents to the traditional content delivery network infrastructure, where the central controller generates bandwidth allocation plans based on the current network state and the local agents deal with the incoming requests in real time. These very interesting works are essentially based on engineering techniques, and they inspire the need of optimization models with more rigorous theoretical analysis.

Complexity of the pricing scheme in the real world also brings challenges to saving bandwidth. The $95^{th}$ percentile billing is a widely used pricing scheme of ISPs, which is considered to reflect the required capacity of the connections in a network better than other methods.
Because most networks are over-provisioned, there is usually reserved balance to handle bursting traffic.
The $95^{th}$ percentile billing refers to ignoring the top $5\%$ samples and charging the bandwidth used by the leaving $95\%$ samples.
To illustrate, for a monthly billing, the samples of consumed bandwidth are sorted at the end of each month, and the top 36 hours (i.e., the top 5\% of the overall 720 hours) of peak traffic will not be charged. It means that the bandwidth could be used for free at a higher rate for up to 72 minutes per day. Therefore, mathematically the cost function generated by $95^{th}$ percentile billing is noncontinuous and nonconvex, and researchers have shown that optimizing the burstable billing of the $95^{th}$ percentile usage is NP-hard, see, e.g. \cite{garey1990,jalaparti2016}.
To save cost charged by the $95^{th}$ percentile billing, a mixed integer linear programming (MILP) model is proposed in \cite{zhan2016optimal}.
In \cite{singh2021cost-effective}, the authors propose another MILP model to further relax this problem. Note that the mentioned characteristics of transmitting data through computer networks are not fully considered in these models.
For the NBA problem to be proposed, we consider the $95^{th}$ percentile billing and the distinctive characteristics of data delivery through networks. Besides, the NBA problem can be scalable to other pricing schemes if the objective function is appropriately adjusted.
Since attributes of data transmission through networks are considered, feasible solutions of the NBA problem have special properties, which will be delineated in the following sections.

\section{Formulation and model}\label{sec:pro_for}

In this section, we propose the NBA problem with details, and formulate it as an integer programming model.

\subsection{Problem identification}\label{sec:pro_ide}

Let the NBA problem be defined on a network $G=(V, E)$ during a billing cycle $P$, where $V = \{1, 2, \dots, n\}$ is the set of network nodes abstracted from servers or data centers; $E = \{(i, j)|i, j \in V, i \neq j\}$ is the set of edges abstracted from network links; and $P$ is a given period of time separated by a set of sampling time points $T = \{1, 2, \dots, p\}$.
Components are given as follows.

\textbf{Definitions:}
\begin{itemize}[leftmargin=*]
\item \textit{bandwidth of an edge:}
the bandwidth of an edge is the amount of data transmitted over it in a given amount of time, which is calculated in megabits per second (Mbps).
\item \textit{ingress bandwidth of a node:} the ingress bandwidth of the node $i \in V$ is the summation of bandwidth of directed edges $\{(u, v)|u\in V, v=i,  (u, v) \in E\}$.
\item \textit{egress bandwidth of a node:} the egress bandwidth of the node $i \in V$ is the summation of bandwidth of directed edges $\{(u, v)|u=i, v\in V, (u, v) \in E\}$.
\item \textit{ingress bandwidth capacity of a node:} the ingress bandwidth capacity of a node is its admissible maximum ingress bandwidth.
\item \textit{egress bandwidth capacity of a node:} the egress bandwidth capacity of a node is its admissible maximum egress bandwidth.
\item \textit{bandwidth allocation plan:} a bandwidth allocation plan refers to a feasible solution of the NBA problem.
\item \textit{optimal bandwidth allocation plan:} the optimal bandwidth allocation plan refers to the optimal solution of the NBA problem.
\end{itemize}

\textbf{Pricing scheme:}

During a billing cycle $P$, the pricing scheme is associated with the bandwidth distribution over $T$.
For the node $i \in V$, let $b_i^{out}$ be the $95^{th}$ percentile of its egress bandwidth distribution $\{\phi_i^{(t)}\}_{t=1}^p$; let $b_i^{in}$ be the $95^{th}$ percentile of its ingress bandwidth distribution $\{\psi_i^{(t)}\}_{t=1}^p$.
The cost incurred on $i$ is given by $u_i \cdot max(\{b_i^{out}, b_i^{in}\})$, where $u_i > 0$ is the unit-price and $max(\cdot)$ is the operation to get the maximum element of a number set.

\textbf{Parameters}:
\begin{itemize}[leftmargin=*]
    \item $u_i > 0$ is the unit-price of bandwidth of the node $i \in V$.
    \item $c_i^{in} > 0$ is the ingress bandwidth capacity of the node $i \in V$.
    \item $c_i^{out} > 0$ is the egress bandwidth capacity of the node $i \in V$.
    \item $E^{(t)} \subseteq E$ is the set of available edges in the time slot $[t, t+1], t \in T$.
    \item $S^{(t)} \subseteq V$ is the set of source nodes providing data in the time slot $[t, t+1], t \in T$.
    \item $D_s^{(t)} \subseteq V \setminus \{s\}$ is the set of destination nodes requiring data that is originally from the source node $s \in S^{(t)}$ in the time slot $[t, t+1], t \in T$.
    \item $w^{(t)}_s > 0$ is the least bandwidth required to deliver data that is originally from the source node $s \in S^{(t)}$ between two nodes in the time slot $[t, t+1], t \in T$.
\end{itemize}

\textbf{Decision variables:}

$f^{(t)}_{s i j} = 1$ if an edge $(i, j) \in E^{(t)}$ is contained in the bandwidth allocation plan to deliver the data that is originally from the source node $s \in S^{(t)}$ in the time slot $[t, t+1], t \in T$; and $0$ otherwise.

\textbf{Objective function:}

The total bandwidth cost during $P$ should be minimized.
Denote the egress bandwidth distribution of the node $i \in V$ over $T$ by
\begin{equation}
    B_i^{out} = \{\sum_{s\in S^{(t)}} (w^{(t)}_s \!\!\cdot\!\!\!\!\!\!\sum_{\substack{(i, j) \in E^{(t)} \\j \in V}}\!\!\!\! f^{(t)}_{s i j})\}^p_{t = 1},
\end{equation}
and the ingress bandwidth distribution of the node $i \in V$ over $T$ by
\begin{equation}
    B_i^{in} = \{\sum_{s\in S^{(t)}} (w^{(t)}_s \!\!\cdot\!\!\!\!\!\!\sum_{\substack{(h, i) \in E^{(t)} \\h \in V}}\!\!\!\! f^{(t)}_{s h i})\}^p_{t = 1}.
\end{equation}

Let $Q_{95}(\cdot)$ be the operation to get the $95^{th}$ percentile of a set of numbers. The goal is represented by
\begin{equation}\label{equ:obj_fun}
    \min \sum_{i \in V}(u_i \cdot max(\{Q_{95}(B_i^{out}), Q_{95}(B_i^{in})\})),
\end{equation}
where $max(\cdot)$ is the operation to get the maximum elements of a number set.

\textbf{Constraints:}
\begin{itemize}[leftmargin=*]
    \item Data can be transmitted out from all source nodes in the time slot $[t, t+1], t \in T$:
\begin{equation}\label{equ:nbap_con_source}
    \sum_{j \in V}f^{(t)}_{s s j} \geq 1, \quad \forall s \in S^{(t)}.
\end{equation}
    \item Data that is originally from a source node can be transmitted into all destination nodes requiring it in the time slot $[t, t+1], t \in T$:
\begin{equation}\label{equ:con_des}
    \sum_{\substack{(i, j) \in E^{(t)} \\i\in V}}\!\!\!\!f^{(t)}_{s i j}\geq 1, \quad \forall s \in S^{(t)}, \forall j \in D_s^{(t)}.
\end{equation}
    \item The egress bandwidth of a node does not exceed its egress bandwidth capacity:
\begin{equation}
    \sum_{s \in S^{(t)}} (w_s^{(t)} \!\!\cdot\!\!\!\!\!\! \sum_{\substack{(i, j) \in E^{(t)} \\ j \in V}} \!\!\!\!f^{(t)}_{s i j}) \leq c_i^{out}, \quad \forall i \in V.
\end{equation}
    \item The ingress bandwidth of a node does not exceed its ingress bandwidth capacity:
\begin{equation}
    \sum_{s \in S^{(t)}} (w_s^{(t)}\!\! \cdot\!\!\!\!\!\! \sum_{\substack{(h, i) \in E^{(t)} \\ h \in V}}\!\!\!\! f^{(t)}_{s h i}) \leq c_i^{in}, \quad \forall i \in V.
\end{equation}
    \item To transmit data that is originally from the same source node, the required ingress bandwidth of a node is no more than the required egress bandwidth unless it is a destination node:
\begin{equation}\label{equ:nbap_con_flow}
    \sum_{\substack{(i, j) \in E^{(t)} \\i \in V}}\!\!\!\!f^{(t)}_{s i j} \leq\!\!\!\! \sum_{\substack{(j, k) \in E^{(t)} \\ k \in V}}\!\!\!\!f^{(t)}_{s j k}, \quad \forall s \in S^{(t)}, \forall j \in V\setminus D^{(t)}_s.
\end{equation}
\end{itemize}
\begin{observation}\label{obs:1}
As data can be replicated within network nodes, it is sufficient to transmit a copy of data into a node by just one edge.
Then, for the optimal bandwidth allocation plan for transmitting data that is originally from the source node $s \in S^{(t)}$ in the time slot $[t, t+1], t \in T$, there is
\begin{equation}\label{equ:no_in_source}
    \sum_{\substack{(i, s) \in E^{(t)} \\i\in V}}\!\!f^{(t)}_{s i s}=0, \quad \forall s \in S^{(t)},
\end{equation}
and
\begin{equation}\label{equ:at_most_one}
    \sum_{\substack{(i, j) \in E^{(t)} \\i\in V}}\!\!f^{(t)}_{s i j}\leq 1, \quad \forall s \in S^{(t)}, \forall j \in V\setminus \{s\}.
\end{equation}
\end{observation}
\begin{proof}
As the node $s \in S^{(t)}$ already contains the transmitted data, (\ref{equ:no_in_source}) holds.
To transmit data that is originally from the source node $s \in S^{(t)}$ in the time slot $[t, t+1], t \in T$, consider the node $j \in V\setminus \{s\}$ and a bandwidth allocation plan. If data is not transmitted over $j$, there is
\begin{equation}\label{equ:no_pass}
    \sum_{\substack{(i, j) \in E^{(t)} \\i\in V}}\!\!f^{(t)}_{s i j} = 0.
\end{equation}
If data is transmitted over $j$, and assume that there is
\begin{equation}\label{equ:m_edges}
    \sum_{\substack{(i, j) \in E^{(t)} \\i\in V}}\!\!f^{(t)}_{s i j} = m,
\end{equation}
where $m \geq 2$. Then, (\ref{equ:m_edges}) indicates that there are $m$ edges directed to the node $j$, which are
denoted by $(i_1, j), (i_2, j), \dots, (i_m, j)$, respectively.
Select the node $i_k, k \in \{1, 2, \dots, m\}$ and delete the edge $(i_k, j)$, data can be still transmitted into the node $j$ through the remaining edges as $m\geq 2$.
Besides, in the time slot $[t, t+1], t \in T$, the egress bandwidth of the node $i_k$ decrease by $w_s^{(t)}$ after deletion.
Also, it holds that
\begin{equation}\label{equ:key_eq}
    Q_{95}(\{\phi_{i_k}^{(1)}, \phi_{i_k}^{(2)}, \dots, \phi_{i_k}^{(t)}-w_s^{(t)}, \dots \phi_{i_k}^{(p)}\}) \leq Q_{95}(\{\phi_{i_k}^{(1)}, \phi_{i_k}^{(2)}, \dots, \phi_{i_k}^{(t)}, \dots \phi_{i_k}^{(p)}\}),
\end{equation}
where $\phi_{i_k}^{(t)}$ is the egress bandwidth of the node $i_{k}$ in the time slot $[t, t+1], t \in T$.
Similarly, in the time slot $[t, t+1], t \in T$, the ingress bandwidth of the node $j$ decreases by $w_s^{(t)}$ after deletion.
Also, there is
\begin{equation}\label{equ:key_eq_2}
    Q_{95}(\{\psi_{i_k}^{(1)}, \psi_{i_k}^{(2)}, \dots, \psi_{i_k}^{(t)}-w_s^{(t)}, \dots \psi_{i_k}^{(p)}\}) \leq Q_{95}(\{\psi_{i_k}^{(1)}, \psi_{i_k}^{(2)}, \dots, \psi_{i_k}^{(t)}, \dots \psi_{i_k}^{(p)}\}),
\end{equation}
where $\psi_{i_k}^{(t)}$ is the ingress bandwidth of node $j$ in the time slot $[t, t+1], t \in T$.
According to (\ref{equ:obj_fun}), (\ref{equ:key_eq}) and (\ref{equ:key_eq_2}), we get an equivalent or better bandwidth allocation plan after deletion.
By induction, there is an equivalent or better bandwidth allocation plan than the current one such that
\begin{equation}\label{equ:pass}
    \sum_{\substack{(i, j) \in E^{(t)} \\i\in V}}\!\!f^{(t)}_{s i j} = 1.
\end{equation}
Combining (\ref{equ:no_pass}) and (\ref{equ:pass}) yields (\ref{equ:at_most_one}).
\end{proof}
According to Observation \ref{obs:1}, (\ref{equ:con_des}) can be reformulated as
\begin{equation}
     \sum_{\substack{(i, j) \in E^{(t)} \\i\in V}}\!\!f^{(t)}_{s i j} = 1, \quad \forall s \in S^{(t)}, \forall j \in D_s^{(t)},
\end{equation}
which means that it is sufficient to transmit a copy of data that is originally from the source node $s \in S^{(t)}$ into its destination nodes $D^{(t)}_s$ by just one edge respectively.
\subsection{Integer programming model}\label{sec:IP model}

With the specifications in Section \ref{sec:pro_ide}, the NBA problem can be formulated as an integer programming model as the following:
\begin{gather}
      \min \sum_{i \in V}u_i \cdot max(\{ Q_{95}(\{\sum_{s\in S^{(t)}} \!\!w^{(t)}_s \!\!\cdot\!\!\!\!\!\!\sum_{\substack{(i, j) \in E^{(t)} \\j \in V}}\!\!\!\! f^{(t)}_{s i j}\}^p_{t = 1}), Q_{95}(\{\sum_{s\in S^{(t)}} \!\!w^{(t)}_s \!\!\cdot\!\!\!\!\!\!\sum_{\substack{(h, i) \in E^{(t)} \\h \in V}}\!\!\!\! f^{(t)}_{s h i}\}^p_{t=1})\}) \nonumber
\end{gather}
\begin{align}
         \text{subject to} \quad\quad\quad\quad\quad\sum_{j \in V}f^{(t)}_{s s j} \geq 1, \quad\quad\quad &\forall s \in S^{(t)}, \forall t \in T,\\[11pt]
          \sum_{\substack{(i, j) \in E^{(t)} \\i\in V}}\!\!\!\!f^{(t)}_{s i j} = 1, \quad\quad\quad &\forall s \in S^{(t)}, \forall j \in D_s^{(t)},  \forall t \in T,\\[11pt]
         \sum_{s \in S^{(t)}} (w_s^{(t)} \!\!\cdot\!\!\!\!\!\! \sum_{\substack{(i, j) \in E^{(t)} \\ j \in V}} \!\!\!\!f^{(t)}_{s i j}) \leq c_i^{out}, \quad\quad &\forall i \in V,  \forall t \in T,\\[11pt]
         \sum_{s \in S^{(t)}} (w_s^{(t)}\!\! \cdot\!\!\!\!\!\! \sum_{\substack{(h, i) \in E^{(t)} \\ h \in V}}\!\!\!\! f^{(t)}_{s h i}) \leq c_i^{in}, \quad\quad &\forall i \in V,  \forall t \in T,\\[11pt]
         \sum_{\substack{(i, j) \in E^{(t)} \\i \in V}}\!\!\!\!f^{(t)}_{s i j} -\!\!\!\! \sum_{\substack{(j, k) \in E^{(t)} \\ k \in V}}\!\!\!\!f^{(t)}_{s j k} \leq 0, \!\!\qquad\quad &\forall s \in S^{(t)}, \forall j \in V\setminus D^{(t)}_s,  \forall t \in T,\\[11pt]
         f^{(t)}_{s i j} \in \{0, 1\}, \quad\quad &\forall s \in S^{(t)}, \forall (i, j) \in E^{(t)},  \forall t \in T.
\end{align}
To deal with this problem, the following propositions may be useful for algorithmic design.
\begin{proposition}\label{prop:no_cyc}
Denote a bandwidth allocation plan for transmitting data that is originally from the source node $s \in S^{(t)}$ in the time slot $[t, t+1], t \in T$, by a graph $G_s^{(t)} = (V_s^{(t)}, E_s^{(t)})$, where $V_s^{(t)} = \{i|\{f_{s i j}^{(t)} \geq 1 \text{ or }f_{s j i}^{(t)} = 1, i, j \in V\}, E_s^{(t)} = \{(i, j)| (i, j) \in E^{(t)}, f_{s i j}^{(t)}=1, i, j \in V_s^{(t)}\}$.
Then, there is no cycle in $G_s^{(t)}$.
\end{proposition}
\begin{proof}
We prove it by contradiction.
Let $\mathcal{G}_s^{(t)}$ be the underlying undirected graph of $G_s^{(t)}$.
Assume that there is a cycle contained in $\mathcal{G}_s^{(t)}$. Without loss of generality, we denote it by
\begin{equation}\label{equ:cyc}
    (i_u, i_{u+1}) \to \cdots \to (i_v, i_{u}).
\end{equation}
Also, we denote by $V_{c}$ the set of nodes contained in the cycle (\ref{equ:cyc}). Then, it follows from (\ref{equ:at_most_one}) that
\begin{equation}\label{equ:pro_pro_1}
    \sum_{\substack{(i, j) \in E^{(t)} \\i\in V}}\!\!\!\!f^{(t)}_{s i j}= 1, \quad \forall j \in V_c.
\end{equation}
By (\ref{equ:no_in_source}) and (\ref{equ:pro_pro_1}), we have $s \notin V_c$, which means that nodes in $V_c$ do not contain transmitted data originally.
Then, to transmit data into the cycle (\ref{equ:cyc}) from $V_a^{(t)}\setminus V_c$, there must be at least one node $j \in V_c$ such that
\begin{equation}
     \sum_{\substack{(i, j) \in E^{(t)} \\i\in V}}\!\!\!\!f^{(t)}_{s i j} > 1,
\end{equation}
which is a contradiction to (\ref{equ:pro_pro_1}).
Therefore, there is no cycle in $\mathcal{G}_s^{(t)}$.
Moreover, there is no cycle in $G_s^{(t)}$.
\end{proof}
\begin{proposition}\label{prop:forest}
Denote a bandwidth allocation plan for transmitting data that is originally from the source node $s \in S^{(t)}$ in the time slot $[t, t+1], t \in T$, by a graph $G_s^{(t)} = (V_s^{(t)}, E_s^{(t)})$, where $V_s^{(t)} = \{i|f_{s i j}^{(t)} \geq 1 \text{ or }f_{s j i}^{(t)} = 1, i, j \in V\}, E_s^{(t)} = \{(i, j)| (i, j) \in E^{(t)}, f_{s i j}^{(t)}=1, i, j \in V_s^{(t)}\}$.
Then, $G_s^{(t)}$ is a directed tree.
\end{proposition}
\begin{proof}
Consider $\mathcal{G}_s^{(t)}$, which is the underlying undirected graph of $G_s^{(t)}$.
By Proposition \ref{prop:no_cyc}, $\mathcal{G}_s^{(t)}$ is an undirected graph in which any two nodes are connected by exactly one path.
Therefore, by the definition of tree, $\mathcal{G}_s^{(t)}$ is an undirected tree.
Moreover, $G_s^{(t)}$ is a directed tree.
\end{proof}

\section{Real applications}\label{rea_ins}

In this section, we show some major real-world cloud computing scenarios that are applications of the NBA problem.
The scenarios include the content delivery network (CDN), the live video delivery network (LVDN), the real-time communication network (RTCN), and the cloud wide area network (Cloud-WAN).

\subsection{Content delivery network}
\subsubsection{Background}
The content delivery network (CDN) is a geographically distributed network of servers.
The goal is to decrease latency and thus improve QoS by distributing the service spatially relative to customers, as a shorter transmission distance usually means lower latency \cite{dilley2002globally}.
\begin{figure}[!ht]
    \centering
    \includegraphics[width=0.76\linewidth]{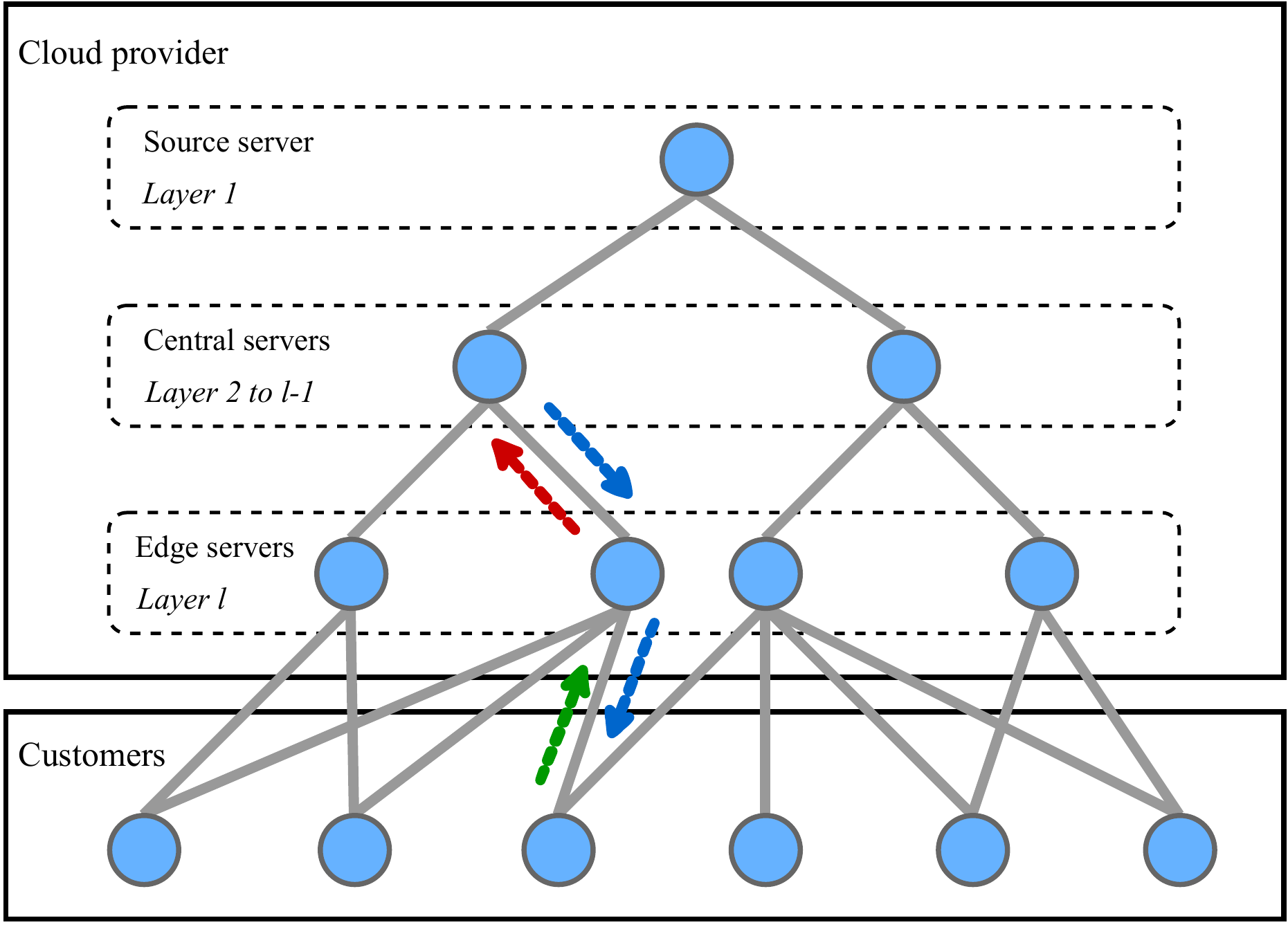}
    \caption{The logical architecture of the content delivery network (CDN).}
    \label{fig:cdn}
\end{figure}
Figure \ref{fig:cdn} shows the logical architecture of the CDN.
For the CDN, there are two major components:
\begin{itemize}[noitemsep]
    \item Server nodes of the cloud provider, whose network is a tree;
    \item Customers that use various devices.
\end{itemize}
For servers of the cloud provider whose network topology is a tree, there are three categories: (1) the source server; (2) central servers; (3) edge servers.
Let the source server be the root of the tree, and there be $l$ layers.
The contents are imported into the CDN through the source server; central servers forward contents layer by layer; and edge servers are responsible for transmitting content to customers.
As contents can be replicated and cached within servers, it is unnecessary to build a transmitting path from the source server to customers at all times.
In other words, it only needs to build a transmitting path from a server that already cached the desired contents to customers.
Therefore, the process to build the transmitting path for a customer can be separated into the following two steps:
\begin{itemize}[noitemsep]
    \item Step 1: build a connection between this customer and an appropriate edge server (the green arrow in Figure \ref{fig:cdn});
    \item Step 2: build connections from the edge server to an upper server recursively until the desired contents are available (the red arrow in Figure \ref{fig:cdn}).
\end{itemize}
Once the transmitting path is built by these two steps, the desired contents can be delivered to the customer through this path (the blue arrow in Figure \ref{fig:cdn}).
For Step 2, the path is deterministic with the given start and end, since any two nodes of a tree are connected by exactly one path.
Therefore, decisions of bandwidth allocation are actually made at Step 1 for edge servers.
The bandwidth allocation of the source server and central servers can be estimated by the bandwidth allocation to edge servers and the occurrence rate of Step 2. In the CDN, the egress bandwidth of a server is usually greater than its ingress bandwidth. Therefore, the amount of bandwidth for the charge is counted by the amount of egress bandwidth.
According to these specific attributes, we extend the NBA problem for the CDN scenario, which is named CDN Bandwidth Allocation (CDN-BA) problem.
\subsubsection{Formulation}
Consider a CDN-BA problem which is defined on $G = (V, E)$ during a billing cycle $P$, where $G$ is a tree; $V = \{1, 2, \dots, n\}$ is the set of servers; $E =\{(i, j)|i, j \in V, i \neq j\}$ is the set of edges abstracted from network links; and $P$ is a given period of time separated by a set of sampling time points $T = \{1, 2, \dots, p\}$.
Components are given as follows.

\textbf{Parameters:}
\begin{itemize}[leftmargin=*]
    \item $V_s \subseteq V$ is the set of the source server and central servers; the source server is the root node of $G$, and central servers are internal nodes of $G$.
    \item $V_e = V \setminus V_s$ is the set of edge servers; the edge servers are leaf nodes of $G$.
    \item $D^{(t)} = \{n+1, n+2, \dots, n+m^{(t)}\}$ is the set of customers in the time slot $[t, t+1], t \in T$.
    \item $E^{(t)} = \{(i, j)|i \in V_e, j \in D^{(t)}\}$ is the set of available connections among edge servers and customers in the time slot $[t, t+1], t \in T$.
    \item $u_i > 0$ is the unit-price of bandwidth of the server $i \in V$.
    \item $c_i > 0$ is the egress bandwidth capacity of the server $i \in V$.
    \item $w^{(t)}_i > 0$ is the least bandwidth required to deliver the contents between an edge server and the customer $i \in D^{(t)}$ in the time slot $[t, t+1], t \in T$.
    \item $r_i^{(t)} \in [0, 1]$ is the probability that the desired content is not cached at the server $i \in V$ in the time slot $[t, t+1], t \in T$.
\end{itemize}
\textbf{Decision variables}:

$f^{(t)}_{i j} = 1$ if an edge $(i, j) \in E^{(t)}$ is contained in the bandwidth allocation plan to transmit data in the time slot $[t, t+1], t \in T$; and $0$ otherwise.

\textbf{Objective function:}

The total bandwidth cost during $P$ should be minimized. The bandwidth cost of edge servers is measured exactly, while the bandwidth cost of central servers and the source server in the time slot $[t, t+1], t \in T$ is estimated through $r_i^{(t)}, i \in V$.
Denote the egress bandwidth distribution of the edge server $i \in V_e$ over $T$ by
\begin{equation}
    B_i = \{\sum_{\substack{(i, j) \in E^{(t)} \\j \in D^{(t)}}} \!\!\!\!w_j^{(t)}\!\! \cdot f^{(t)}_{i j}\}^p_{t = 1}.
\end{equation}
Let $Q_{95}(\cdot)$ be the operation to get the $95^{th}$ percentile of a set of numbers. The cost of all edge severs is represented by
\begin{equation}
    \sum_{i \in V_e}u_i \cdot Q_{95}(B_i).
\end{equation}
Then, the bandwidth of central servers and the source server can be estimated layer by layer.
For example, let $k \in V_s$ be a server node whose children are edge server nodes. Then, the egress bandwidth of $k$ in the time slot $[t, t+1], t \in T$ is estimated by
\begin{equation}
    \sum_{(k, i) \in E} (r_i^{(t)} \!\!\cdot\!\!\!\!\!\! \sum_{\substack{(i, j) \in E^{(t)} \\j \in D^{(t)}}} \!\!w_j^{(t)}\!\! \cdot f^{(t)}_{i j}).
\end{equation}
For simplicity, we denote the egress bandwidth of $k \in V_s$ in the time slot $[t, t+1], t \in T$ by $b_k^{(t)}$, which is obtained by an operation $R$ given the egress bandwidth of all edge servers
and $\{r_i^{(t)}\}_{i \in V}$:
\begin{equation}
    b_k^{(t)} = R(k, \{\sum_{\substack{(i, j) \in E^{(t)} \\j \in D^{(t)}}}\!\!\!\! w_j^{(t)}\!\! \cdot f^{(t)}_{i j}\}_{i \in V_e}, \{r_i^{(t)}\}_{i \in V}).
\end{equation}
Then, the total cost of central servers and the source server is denoted by
\begin{equation}
    \sum_{k \in V_s}u_k \cdot Q_{95}(\{b_k^{(t)}\}_{t=1}^{p}).
\end{equation}
Therefore, the goal is
\begin{equation}
    \min \sum_{i \in V_e}u_i \cdot Q_{95}(B_i) + \sum_{k \in V_s}u_k \cdot Q_{95}(\{b_k^{(t)}\}_{t=1}^{p}).
\end{equation}

\textbf{Constraints:}
\begin{itemize}[leftmargin=*]
    \item Data can be delivered into all customer nodes from edge servers in the time slot $[t, t+1], t \in T$:
\begin{equation}
    \sum_{\substack{(i, j) \in E^{(t)} \\i\in V_e}}\!\!\!\!f^{(t)}_{i j} = 1, \quad \forall j \in D^{(t)}.
\end{equation}
    \item The egress bandwidth of a server node does not exceed its egress bandwidth capacity:
\begin{equation}
     \sum_{\substack{(i, j) \in E^{(t)} \\ j \in D^{(t)}}}\!\!\!\!w_j^{(t)} f^{(t)}_{i j}\leq c_i, \quad \forall i \in V_e;
\end{equation}
\begin{equation}
     R(k, \{\sum_{\substack{(i, j) \in E^{(t)} \\j \in D^{(t)}}}\!\!\!\! w_j^{(t)} \!\!\cdot f^{(t)}_{i j}\}_{i \in V_e}, \{r_i^{(t)}\}_{i \in V})\leq c_i, \quad \forall k \in V_s.
\end{equation}
\end{itemize}

\begin{remark}
For the CDN, there are some servers that have already cached the desired content. Once the connection between a customer and an edge server is built, a path to transmit desired contents is established recursively.
Therefore, the constraint (\ref{equ:nbap_con_source}) in the NBA problem is always satisfied and thus it can be omitted in the CDN scenario.
\end{remark}
\begin{remark}
As the structure of the CDN is a tree, by Proposition \ref{prop:forest}, the constraint (\ref{equ:nbap_con_flow}) in the NBA problem is always satisfied and thus it can also be omitted in the CDN scenario.
\end{remark}

Based on the discussions above, the CDN-BA problem is formulated as
\begin{equation}
    \begin{aligned}
    \min \quad &\sum_{i \in V_e}u_i \cdot Q_{95}( \{\sum_{\substack{(i, j) \in E^{(t)} \\j \in D^{(t)}}} \!\!\!\!w_j^{(t)}\!\! \cdot f^{(t)}_{i j}\}^p_{t = 1}) \\
    &+ \sum_{k \in V_s}u_k \cdot Q_{95}(\{R(k, \{ \sum_{\substack{(i, j) \in E^{(t)} \\j \in D^{(t)}}} \!\!\!\!w_j^{(t)} \!\!\cdot f^{(t)}_{i j}\}_{i \in V_e}, \{r_i^{(t)}\}_{i \in V})\}_{t=1}^{p}) \nonumber
    \end{aligned}
\end{equation}
\begin{align}
    \text{subject to}\quad\quad\quad\quad\quad\!\!\!\sum_{\substack{(i, j) \in E^{(t)} \\i\in V_e}}f^{(t)}_{i j} = 1, \quad\quad\quad\quad\quad\quad &\forall j \in D^{(t)}, \forall t \in T, \\
    \sum_{\substack{(i, j) \in E^{(t)} \\ j \in D^{(t)}}}\!\!\!\!w_j^{(t)} f^{(t)}_{i j}\leq c_i, \quad\quad\quad\quad\quad\quad\!\! &\forall i \in V_e, \forall t \in T, \\
    R(k, \{\sum_{\substack{(i, j) \in E^{(t)} \\j \in D^{(t)}}}\!\!\!\!w_j^{(t)} \!\!\cdot f^{(t)}_{i j}\}_{i \in V_e}, \{r_i^{(t)}\}_{i \in V})\leq c_k, \quad &\forall k \in V_s, \forall t \in T,\\
    f^{(t)}_{i j} \in \{0, 1\}, \quad\quad\quad\quad\quad\quad &\forall (i, j) \in E^{(t)}, \forall t \in T.
\end{align}

\subsection{Live video delivery network}

\subsubsection{Background}
Live video means an online video that is recorded and transmitted over a network in real-time.
Delivery of live video is challenging because of the unavailability of cache technology and the low latency requirement.
With the increment of network bandwidth, many applications related to live video delivery have gradually become into reality such as live news, live shows, and live courses.
Recent surveys in \cite{ globe2017, elements2019} have shown that Internet users had watched an accumulation of 1.1 billion hours of live video in 2019 and the traffic of live video had reached 82$\%$ of overall Internet traﬃc by the end of 2020. Also, it is estimated in \cite{news2021} that the video streaming market will hit 223.98 billion dollars by 2028.
The diversity and volume of live video bring not only business opportunities, but also scientific challenges to save bandwidth cost for cloud providers.
\begin{figure}[!ht]
    \centering
    \includegraphics[width=0.76\linewidth]{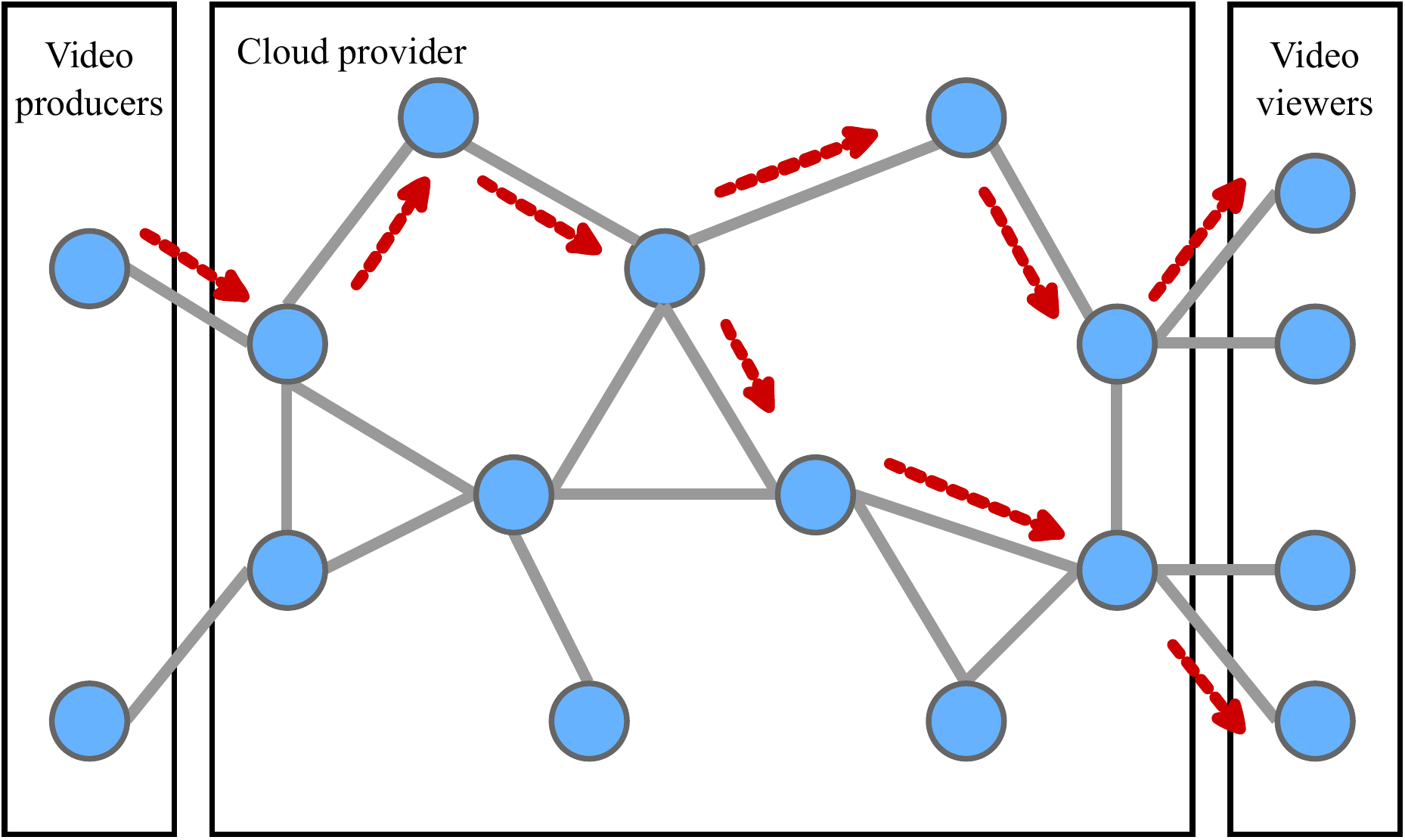}
    \caption{The logical architecture of the live video delivery network (LVDN).}
    \label{fig:lvdn}
\end{figure}

Figure \ref{fig:lvdn} shows the logical architecture of the live video delivery network (LVDN).
There are three major stakeholders in the LVDN: (1) video producers; (2) the cloud provider; (3) video viewers.
Video producers create live videos in various roles such as live news reporters, actors in live shows, and teachers in live courses.
Cloud providers such as Amazon, Microsoft, Alibaba, and Huawei pay to ISPs to deliver live video from video producers to viewers.
For example, red arrows in Figure \ref{fig:lvdn} represent paths to transmit a live video from a video producer to viewers through the cloud provider's servers.
In the process of live video delivery, video producers and viewers demand high quality of video (e.g., video resolution, video frame rate, color depth), instant accessibility, and low buffering ratios, while the cloud provider targets to provide QoS guaranteed transmission service with minimized bandwidth cost.

To cast the live video delivery problem into the NBA problem, a video producer in the LVDN is a source node, and the related viewers are destination nodes.
It means that the servers of cloud providers are intermediary nodes, and the ingress bandwidth of them is always no more than their egress bandwidth considering the data replication ability of servers.
Therefore, the amount of bandwidth for the charge is counted by the amount of egress bandwidth.
Besides, a video producer is admissible to access only one server to upload
video in this scenario.
By grouping video producers and viewers appropriately, we extend the NBA problem for the LVDN scenario, which is named LVDN Bandwidth Allocation (LVDN-BA) Problem.

\subsubsection{Formulation}
Consider a LVDN-BA problem which is defined on $G = (V, E)$ during a billing cycle $P$, $V = \{1, 2, \dots, n\}$ is the set of servers; $E = \{(i, j)|i, j \in V, i \neq j\}$ is the set of edges abstract from inner-domain network links; and $P$ is a given period of time separated by a set of sampling time points $T = \{1, 2, \dots, p\}$.
Components are given as follows.

\textbf{Parameters:}
\begin{itemize}[leftmargin=*]
    \item $S^{(t)} = \{n+1, n+2, \dots, n+m^{(t)}\}$ is the set of video producers in the time slot $[t, t+1], t \in T$.
    \item $D_s^{(t)} \subseteq \mathbb{Z}^{+}\setminus(V \cup S^{(t)})$ is the set of viewers of the video producer $s \in S^{(t)}$ in the time slot $[t, t+1], t \in T$, where $\mathbb{Z}^{+}$ is the set of positive integers.
    \item $V^{(t)}$ represents $V\cup S^{(t)}\cup (\bigcup_{s \in S^{(t)}}D_s^{(t)})$.
    \item $E^{(t)} = \{(i, j)|i, j \in V^{(t)}, i \neq j\}$ is the set of available connections among servers, video producers, and video viewers in the time slot $[t, t+1], t \in T$.
    \item $u_i > 0$ is the unit-price of bandwidth of the server $i \in V$.
    \item $c_i > 0$ is the egress bandwidth capacity of the server $i \in V$.
    \item $w^{(t)}_s > 0$ is the least bandwidth required to deliver live video created by video producer $s \in S^{(t)}$ in the time slot $[t, t+1], t \in T$.
\end{itemize}
\textbf{Decision variables}:

$f^{(t)}_{s i j} = 1$ if an edge $(i, j) \in E^{(t)}$ is contained in the bandwidth allocation plan to deliver the live video created by the video producer $s \in S^{(t)}$ in the time slot $[t, t+1], t \in T$; and $0$ otherwise.
\newpage
\textbf{Objective function:}

The total bandwidth cost during $P$ should be minimized.
Denote the egress bandwidth distribution of node $i \in V$ over $T$ by
\begin{equation}
    B_i = \{\sum_{s \in S^{(t)}}(w_s^{(t)}\!\! \cdot\!\!\!\!\!\!\sum_{\substack{(i, j) \in E^{(t)} \\j \in V^{(t)}}} \!\!\!\!f^{(t)}_{s i j})\}^p_{t = 1}.
\end{equation}
Let $Q_{95}(\cdot)$ be the operation to get the $95^{th}$ percentile of a set of numbers. The goal is represented by
\begin{equation}
    \min \sum_{i \in V}u_i \cdot Q_{95}(B_i).
\end{equation}

\textbf{Constraints:}
\begin{itemize}[leftmargin=*]
    \item Data can be transmitted out from video producers in the time slot $[t, t+1], t \in T$:
\begin{equation}
    \sum_{j \in V}f^{(t)}_{s s j} = 1, \quad \forall s \in S^{(t)}.
\end{equation}
Note that a video producer is admissible to access only one server to upload video in this scenario.
    \item Data can be transmitted into all viewer nodes in the time slot $[t, t+1], t \in T$:
\begin{equation}
    \sum_{\substack{(i, j) \in E^{(t)} \\i\in V}}\!\!\!\!f^{(t)}_{s i j} = 1, \quad\forall s \in S^{(t)}, \forall j \in D_s^{(t)}.
\end{equation}
    \item The egress bandwidth of a server node does not exceed its egress bandwidth capacity:
\begin{equation}
    \sum_{s \in S^{(t)}}(w_s^{(t)}\!\! \cdot\!\!\!\!\!\!\sum_{\substack{(i, j) \in E^{(t)} \\j \in V^{(t)}}}\!\!\!\! f^{(t)}_{s i j})\leq c_i, \quad \forall i \in V.
\end{equation}
    \item The ingress bandwidth of a server node in the path to transmit data that is originally from the same video producer is no more than its egress bandwidth:
\begin{equation}
    \sum_{\substack{(i, j) \in E^{(t)} \\i\in V^{(t)}}}\!\!\!\!f^{(t)}_{s i j} \leq\!\!\!\! \sum_{\substack{(j, k) \in E^{(t)} \\ k \in V^{(t)}}}\!\!\!\!f^{(t)}_{s j k}, \quad \forall s \in S^{(t)}, \forall j \in V.
\end{equation}
\end{itemize}

Based on the discussions above, the LVDN-BA problem is formulated as
\begin{gather}
        \min \sum_{i \in V}u_i \cdot Q_{95}(\{\sum_{s\in S^{(t)}}(w_s^{(t)}\!\! \cdot\!\!\!\!\!\!\sum_{\substack{(i, j) \in E^{(t)} \\j \in V^{(t)}}}\!\!\!\! f^{(t)}_{s i j})\}^p_{t = 1}) \nonumber
\end{gather}
\begin{align}
         \text{subject to}\quad\quad\quad\quad\quad\sum_{j \in V}f^{(t)}_{s s j} = 1, \quad\quad\quad & \forall s \in S^{(t)}, \forall t \in T, \\
         \sum_{\substack{(i, j) \in E^{(t)} \\i\in V}}\!\!\!\!f^{(t)}_{s i j} = 1, \quad\quad\quad &\forall s \in S^{(t)}, \forall j \in D_s^{(t)}, \forall t \in T, \\
          \sum_{s \in S^{(t)}}(w_s^{(t)}\!\! \cdot\!\!\!\!\!\!\sum_{\substack{(i, j) \in E^{(t)} \\j \in V^{(t)}}}\!\!\!\! f^{(t)}_{s i j})\leq c_i, \quad\quad\quad\!\! & \forall i \in V, \forall t \in T,\\
         \sum_{\substack{(i, j) \in E^{(t)} \\i\in V^{(t)}}}\!\!\!\!f^{(t)}_{s i j} -\!\!\!\! \sum_{\substack{(j, k) \in E^{(t)} \\ k \in V^{(t)}}}\!\!\!\!f^{(t)}_{s j k} \leq 0, \quad\quad\quad\!\! & \forall s \in S^{(t)}, \forall j \in V, \forall t \in T, \\
         f^{(t)}_{s i j } \in \{0, 1\}, \quad\quad\quad\!\!\!\! & \forall s \in S^{(t)}, \forall (i, j) \in E^{(t)}, \forall t \in T.
\end{align}
\begin{remark}
The process of live video delivery can be separated into two steps: (1) build connections between customers (video producers and viewers) and servers; (2) build a path with the given start and end among servers. These two steps are intrinsically related, while we may consider solving them individually in a ``divide-and-conquer'' manner because it is too challenging to solve them as a whole.
\end{remark}

\subsection{Real-time communication network}

\subsubsection{Background}
Real-time communication (RTC) is a mode of interchanging information, where participants exchange information instantly or with negligible latency.
The transmitted contents include audio, videos, texts, files, and so on.
Developments of RTC technologies have boosted various application scenarios into reality, e.g., video chats, live conferences, smart factory, and cloud gaming. Many cloud computing business have been thus created. For cloud providers supplying RTC service, saving the bandwidth cost is also a major task, along with other purposes such as increasing the reliability and capability of the network transmission.
\begin{figure}[!ht]
    \centering
    \includegraphics[width=0.76\linewidth]{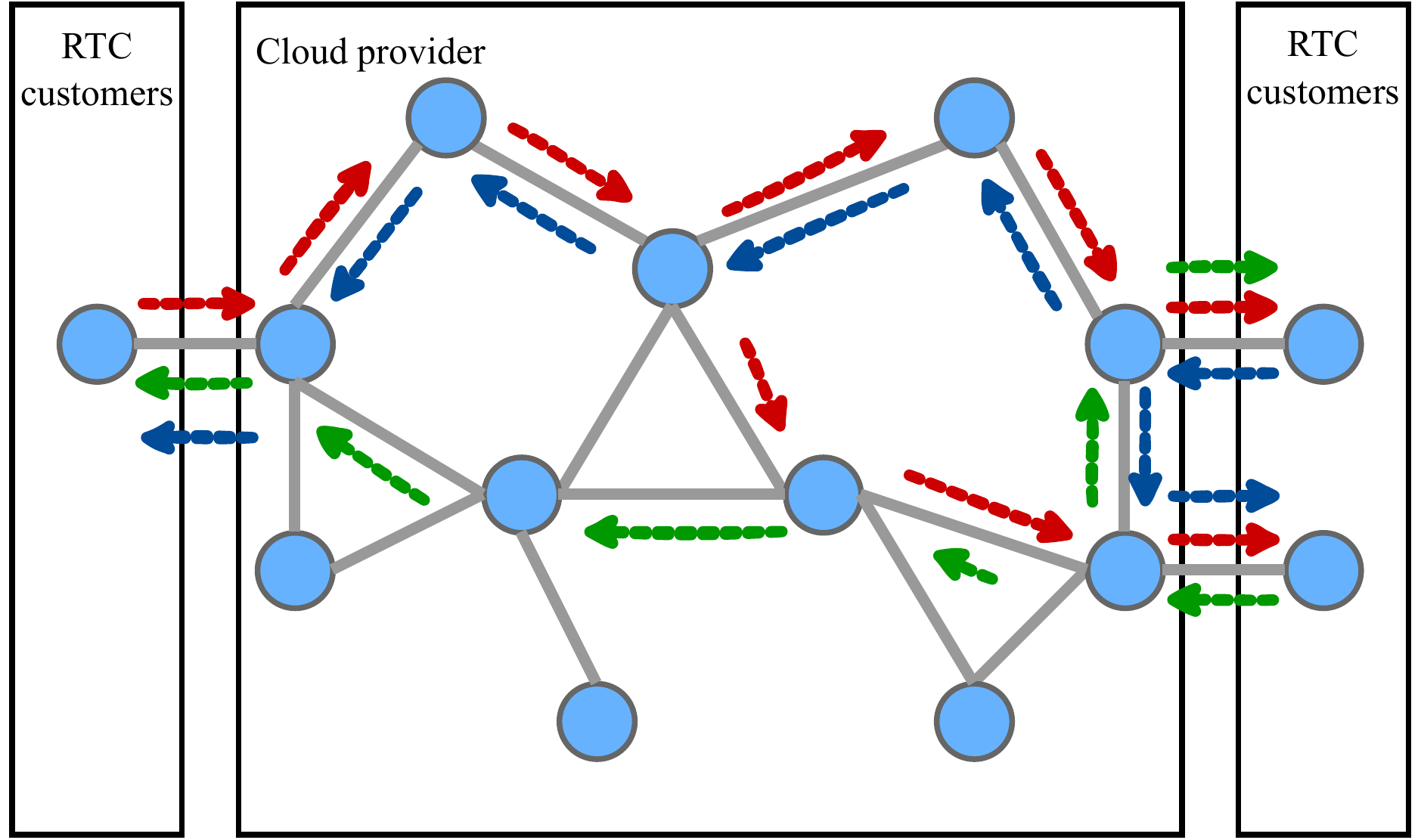}
    \caption{The logical architecture of the real-time communication network (RTCN).}
    \label{fig:rtc}
\end{figure}

Figure \ref{fig:rtc} shows the logical architecture of the real-time communication network (RTCN).
As shown in Figure \ref{fig:rtc}, three customers are geographically distributed.
Assume that all of them are in the same RTC group (e.g., a live conference), the cloud provider needs to build peer-to-peer paths for the customers to exchange information.
In Figure \ref{fig:rtc}, example paths are labeled by arrows with different colors.
Besides, an RTC participant is admissible to access only one server to upload contents in this scenario.
In analogy with LVDN, RTC participants play roles in both video producers and viewers.
Therefore, the LVDN-BA problem can be applied in an RTC group by regarding every participant as a video producer while the remaining participants viewers. According to this idea, we extend the NBA problem for the RTCN scenario, which is named RTCN Bandwidth Allocation (RTCN-BA) problem.

\subsubsection{Formulation}
Consider an RTCN-BA problem which is defined on $G = (V, E)$ during a billing cycle $P$, where $V = \{1, 2, \dots, n\}$ is the set of servers; $E = \{(i, j)|i, j \in V, i \neq j\}$ is the set of edges abstracted from network links; and $P$ is a given period of time separated by a set of sampling time points $T = \{1, 2, \dots, p\}$.
Components are given as follows.

\textbf{Parameters:}
\begin{itemize}[leftmargin=*]
    \item $S^{(t)} = \{n+1, n+2, \dots, n+m^{(t)}\}$ is the set of all RTC participants in the time slot $[t, t+1], t \in T$.
    \item $A^{(t)} = \{A_1^{(t)}, A_2^{(t)}, \dots, A_{g^{(t)}}^{(t)}\}$,  where $A_1^{(t)}, A_2^{(t)}, \dots, A_{g^{(t)}}^{(t)} \subseteq S^{(t)}$ are RTC groups in the time slot $[t, t+1], t \in T$.
    \item $V^{(t)}$ represents $V \cup S^{(t)}$.
    \item $E^{(t)} = \{(i, j)|i, j \in V^{(t)}, i\neq j\}$ is the set of available connections among RTC participants and servers in the time slot $[t, t+1], t \in T$.
    \item $u_i > 0$ is the unit-price of bandwidth of the server $i \in V$.
    \item $c_i > 0$ is the egress bandwidth capacity of the server $i \in V$.
    \item $w^{(t)}_s > 0$ is the least bandwidth required to deliver contents created by the participant $s \in S^{(t)}$ in the time slot $[t, t+1], t \in T$.
\end{itemize}

\textbf{Decision Variables:}

$f^{(t)}_{s i j} = 1$ if an edge $(i, j) \in E^{(t)}$ is contained in the bandwidth allocation plan to deliver contents created by the participant $s \in S^{(t)}$ in the time slot $[t, t+1], t \in T$; and $0$ otherwise.

\textbf{Objective Function:}

The total bandwidth cost during the billing cycle $P$ should be minimized.
Denote the egress bandwidth distribution of node $i \in V$ over $T$ by
\begin{equation}
    B_i = \{\sum_{s \in S^{(t)}}(w_s^{(t)}\!\! \cdot\!\!\!\!\!\!\sum_{\substack{(i, j) \in E^{(t)} \\j \in V^{(t)}}}\!\!\!\! f^{(t)}_{s i j})\}^p_{t = 1}.
\end{equation}
Let $Q_{95}(\cdot)$ be the operation to get the $95^{th}$ percentile of a set of numbers. The goal is represented by
\begin{equation}
    \min \sum_{i \in V}u_i \cdot Q_{95}(B_i).
\end{equation}
\textbf{Constraints:}
\begin{itemize}[leftmargin=*]
    \item Data can be transmitted out from all RTC participants in the time slot $[t, t+1], t \in T$:
\begin{equation}
    \sum_{j \in V}f^{(t)}_{s s j} = 1, \quad \forall s \in S^{(t)}.
\end{equation}
Note that an RTC participant is admissible to access only one server to upload contents.
    \item Data can be transmitted to other participants in an RTC group:
\begin{equation}
    \sum_{\substack{(i, j) \in E^{(t)} \\i\in V}}\!\!\!\!f^{(t)}_{s i j} = 1, \quad \forall A_g \in A^{(t)}, \forall s \in A_g, \forall j \in A_g \setminus \{s\}.
\end{equation}
    \item The egress bandwidth of a node does not exceed its egress bandwidth capacity:
\begin{equation}
    \sum_{s \in S^{(t)}}(w_s^{(t)} \!\!\cdot\!\!\!\!\!\!\sum_{\substack{(i, j) \in E^{(t)} \\j \in V^{(t)}}}\!\!\!\! f^{(t)}_{s i j})\leq c_i, \quad \forall i \in V.
\end{equation}
    \item The ingress bandwidth of a server node in the path to transmit data that is originally from the same RTC participant is no more than its egress bandwidth:
\begin{equation}
    \sum_{\substack{(i, j) \in E^{(t)} \\i\in V^{(t)}}}\!\!\!\!f^{(t)}_{s i j} \leq\!\!\!\! \sum_{\substack{(j, k) \in E^{(t)} \\ k \in V^{(t)}}}\!\!\!\!f^{(t)}_{s j k}, \quad \forall s \in S^{(t)}, \forall j \in V.
\end{equation}
\end{itemize}

Based on the discussions above, the RTCN-BA problem is formulated as
\begin{gather}
        \min \sum_{i \in V}u_i \cdot Q_{95}(\{\sum_{s\in S^{(t)}}(w_s^{(t)}\!\! \cdot\!\!\!\!\!\!\sum_{\substack{(i, j) \in E^{(t)} \\j \in V^{(t)}}} \!\!\!\!f^{(t)}_{s i j})\}^p_{t = 1}) \nonumber
\end{gather}
\begin{align}
         \text{subject to}\quad\quad\quad\quad\quad\sum_{j \in V}f^{(t)}_{s s j} = 1, \quad & \forall s \in S^{(t)}, \forall t \in T, \\[11pt]
          \sum_{\substack{(i, j) \in E^{(t)} \\i\in V}}\!\!\!\!f^{(t)}_{s i j} = 1, \quad & \forall A_g \in A^{(t)}, \forall s \in A_g, \forall j \in A_g \setminus \{s\}, \forall t \in T, \\[11pt]
          \sum_{s \in S^{(t)}}(w_s^{(t)}\!\! \cdot\!\!\!\!\!\!\sum_{\substack{(i, j) \in E^{(t)} \\j \in V^{(t)}}} \!\!\!\!f^{(t)}_{s i j})\leq c_i, \quad\!\! & \forall i \in V, \forall t \in T,\\[11pt]
          \sum_{\substack{(i, j) \in E^{(t)} \\i\in V^{(t)}}}\!\!\!\!f^{(t)}_{s i j} -\!\!\!\! \sum_{\substack{(j, k) \in E^{(t)} \\ k \in V^{(t)}}}\!\!\!\!f^{(t)}_{s j k} \leq 0, \quad\!\! & \forall s \in S^{(t)}, \forall j \in V, \forall t \in T, \\[11pt]
         f^{(t)}_{s i j } \in \{0, 1\}, \quad\!\!\!\! & \forall s \in S^{(t)}, \forall (i, j) \in E^{(t)}, \forall t \in T.
\end{align}

\subsection{Cloud wide area network}
\subsubsection{Background}

A wide area network (WAN) refers to a telecommunication network that extends over a large geographic area.
The cloud wide area network (Cloud-WAN) enables cloud providers to supply services for clients that are geographically distributed.
The rapidly increasing demands have inspired cloud providers to focus on the traffic exchanged between the Cloud-WAN and other networks on the Internet, see, e.g. \cite{singh2021cost-effective}.
In cloud computing, clients access the service from the cloud
through Point of Presence (PoP).
In general, the requirement of a client is responded by the PoP which is geographically closest to the client, because a shorter distance usually means lower latency and higher quality. However, to balance the bandwidth utilization and thus save the cost, cloud providers may also redirect some requirements to a farther PoP that can also provide service with acceptable QoS.

\begin{figure}[!ht]
    \centering
    \includegraphics[width=0.76\linewidth]{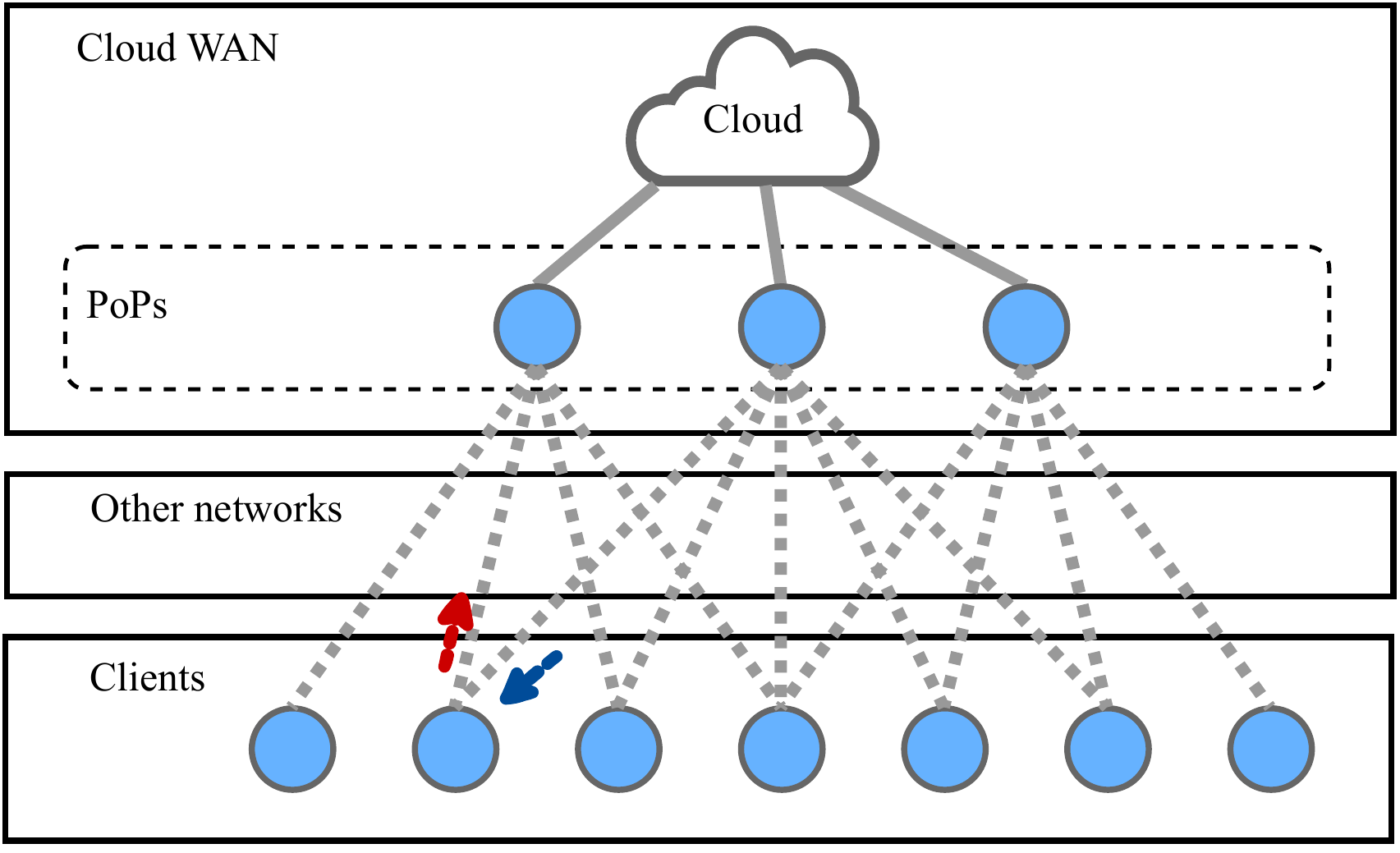}
    \caption{The logical architecture of the cloud wide area network (Cloud-WAN).}
    \label{fig:cwan}
\end{figure}
Figure \ref{fig:cwan} shows the logical architecture of the Cloud-WAN.
As shown in Figure \ref{fig:cwan},
PoPs access the cloud through inner-domain networks, while clients access PoPs through inter-domain networks.
As shown by the red arrow and the blue arrow, the PoP that responds to the requirement of a client is different from the PoP that received it before, which is controlled by the traffic scheduling system of the cloud provider to achieve some central-level goals such as balancing traffic and saving bandwidth cost.
Besides, the traffic demand of a client can be transferred by more than one PoP.

As the traffic transmitted into PoPs is not controlled by the cloud provider, only the egress bandwidth of PoPs is considered in a traffic scheduling system.
According to the NBA problem, PoPs in a Cloud-WAN are source nodes, while clients are destination nodes.
The structure of data transmission network is a bipartite graph.
By the background of Cloud-WAN, a destination node may receive data from more than one source node.
According to these specific attributes, we extend the NBA problem to this scenario, which is named Cloud-WAN Bandwidth Allocation (Cloud-WAN-BA) problem.

\subsubsection{Formulation}
Consider a Cloud-WAN-BA problem which is defined on $G = (V, E)$ during a billing cycle $P$, where $G$ is a bipartite graph;
$V_d = \{1, 2, \dots, m\}$ is the set of PoPs;
$V_r = \{m+1, m+2, \dots, n\}$ is the set of clients, $V = V_d \cup V_r$;
$E = \{(i, j)|i \in V_d, j \in V_r\}$ is the set of edges abstracted from network links;
and $P$ is a given period of time separated by a set of sampling time points $T = \{1, 2, \dots, p\}$.
Components are given as follows.

\textbf{Parameters:}
\begin{itemize}[leftmargin=*]
    \item $u_i > 0$ is the unit-price of bandwidth of the PoP $i \in V_d$.
    \item $d^{(t)}_j>0$ is the total traffic demands of the client $j\in V_r$ in the time slot $[t, t+1], t\in T$.
    \item $c_i>0$ is the egress bandwidth capacity of the PoP $i\in V_d$.
    \item $E^{(t)} = \{(i, j)|i \in V_d, j \in V_r\}$ is the set of available connections among PoPs and clients in the time slot $[t, t+1], t \in T$.
\end{itemize}

\textbf{Decision Variables:}

$f^{(t)}_{ i j }\in \mathbb{Z}^+$ represents the amount of traffic assigned to the edge $(i,j)\in E^{(t)}$ in the time slot $[t, t+1], t\in T$, where $\mathbb{Z}^+$ is the set of positive integers.

\textbf{Objective function:}

The total bandwidth cost during the billing cycle $P$ should be minimized.
As only egress bandwidth is considered, we denote the egress bandwidth distribution of node $i \in V_d$ over $T$ by
\begin{equation}
    B_i = \{\sum_{\substack{(i, j) \in E^{(t)} \\j \in V_r}}\!\! f^{(t)}_{i j}\}^p_{t = 1}.
\end{equation}
Let $Q_{95}(\cdot)$ be the operation to get the $95^{th}$ percentile of a set of numbers. The goal is represented by
\begin{equation}
    \min \sum_{i \in V_d}u_i \cdot Q_{95}(B_i).
\end{equation}

\textbf{Constraints:}
\begin{itemize}[leftmargin=*]
    \item The traffic demands of a client have to be fulfilled in the time slot $[t,t+1],\forall t\in T$:
\begin{equation}
    \sum_{\substack{(i, j) \in E^{(t)} \\i\in V_d}}\!\!\!\!f^{(t)}_{ i j }=d^{(t)}_j, \quad \forall j \in V_r,
\end{equation}
    \item The egress bandwidth of the PoP $i\in V_d$ does not exceed its capacity $c_i$:
\begin{equation}
    \sum_{\substack{(i, j) \in E^{(t)} \\ j \in V_r}}\!\!\!\!f^{(t)}_{ i j } \leq c_i, \quad\forall i \in V_d,
\end{equation}
\end{itemize}
Based on the discussions above, the Cloud-WAN-BA problem is formulated as
\begin{gather}
    \min \sum_{i \in V_d}u_i \cdot Q_{95}(\{\sum_{\substack{(i, j) \in E^{(t)} \\j \in V_r}} \!\!\!\!f^{(t)}_{i j}\}^p_{t = 1}) \label{pcN obj} \nonumber
\end{gather}
\begin{align}
    \text{subject to}\quad\quad\quad\quad\quad\sum_{\substack{(i, j) \in E^{(t)} \\i\in V_d}}\!\!\!\!f^{(t)}_{ i j }=d^{(t)}_j, \quad\quad\quad &\forall j \in V_r, \forall t \in T, \label{equ:cwan1} \\[11pt]
    \sum_{\substack{(i, j) \in E^{(t)} \\ j \in V_r}}\!\!\!\!f^{(t)}_{ i j }\leq c_i, \quad\quad\quad\quad\!\! &\forall i \in V_d, \forall t \in T, \label{equ:cwan2}\\[11pt]
    f^{(t)}_{ i j } \in \mathbb{Z}^+, \quad\quad\quad &\forall (i, j) \in E^{(t)}, \forall t \in T. \label{pcn: variable}
\end{align}

Considering the structure of (\ref{equ:cwan1}) and (\ref{equ:cwan2}), we quote the following proposition.

\begin{proposition}[totally unimodular \cite{schrijver1998theory}]
A matrix A is said to be totally unimodular if the determinant of every square submatrix formed from it has value $0$, $+1$, or $-1$.
In the system of equations $Ax = b$, assume that $A$ is totally unimodular and that all elements of $A$ and $b$ are integers.
Then, all basic solutions have integer components.
\end{proposition}

\section{Discussions}\label{sec:dis}

We propose the Network Bandwidth Allocation (NBA) problem for cloud computing, and formulate it as an integer programming model. The following  three distinctive characteristics of transmitting data through networks are considered: (1) the cost functions generated by pricing schemes are nonlinear, nonconvex and noncontinuous; (2) network topology updates periodically; (3) data can be replicated within network nodes. These attributes, along with the high dimensionality of variables, make it very challenging to mathematically analyze the NBA problem.
The proposed NBA problem is a fundamental and high-level representation for modeling transmitting data through networks while minimizing the bandwidth cost, and it can be extended flexibly to suit various cloud computing scenarios. Moreover, the NBA problem can be easily modified to suit for other pricing schemes if its objective function based on the $95^{th}$ percentile billing is appropriately adjusted.
We show four real cloud computing applications of the NBA problem: the content delivery network (CDN),
the live video delivery network (LVDN), the real-time communication network (RTCN), and the cloud wide area network (Cloud-WAN). It is expected that our first effort of nailing down the mathematical models for these cloud computing problems can boost more rigorous and insightful studies from various perspectives.

It is interesting yet very challenging to design efficient algorithms which can be applied to the integer programming formulation of the NBA problem. Standard solvers such as \textit{SCIP}, \textit{CPLEX}, and \textit{Gurobi} turn out not to work well due to the underlying hierarchical structures, combinatorial properties, as well as the high dimensionality of variables of the corresponding integer programming formulation. We are exploring the integration of conventional optimization techniques/algorithms with problem-tailored heuristics to design applicable and/or efficient algorithms. It is also noted that some applications usually require solving a large set of instances that are generated by homogeneous datasets yet with different problem parameters. Hence, it is promising to apply some machine-learning based methods to extract valuable information from data. For example, bandwidth consumption empirically follows a regular distribution over the billing cycles because there is a strong relationship between customer usage habits and time over billing cycles.
Besides, since the NBA problem is defined on networks, the graph neural network (GNN) seems to be a primary deep learning architecture that can be used to alleviate the difficulties in solving the NBA problem via its functions of learning, reasoning, and generalizing graph-structured data.


\bibliographystyle{ieeetr}
\bibliography{reference}

\end{document}